\documentclass[a4paper,12pt]{amsart}

\usepackage{color}

\theoremstyle{plain}
\newtheorem{thm}{Theorem}[section]
\newtheorem{theorem}[thm]{Theorem}

\newtheorem{lemma}[thm]{Lemma}

\theoremstyle{definition}
\newtheorem{remark}[thm]{Remark}

\newtheorem{defin}[thm]{Definition}

\newtheorem{example}[thm]{Example}

\numberwithin{equation}{section}

\newcommand{\sA}{{\mathcal A}}
\newcommand{\sB}{{\mathcal B}}
\newcommand{\sC}{{\mathcal C}}

\newcommand{\sF}{{\mathcal F}}

\newcommand{\sH}{{\mathcal H}}

\newcommand{\sL}{{\mathcal L}}

\newcommand{\sR}{{\mathcal R}}
\newcommand{\sS}{{\mathcal S}}

\newcommand{\sW}{{\mathcal W}}

\newcommand{\sY}{{\mathcal Y}}
\newcommand{\sZ}{{\mathcal Z}}

\newcommand{\PP}{\ensuremath{\mathbb{P}}}

\newcommand{\CC}{\ensuremath{\mathbb{C}}}

\newcommand{\ZZ}{\ensuremath{\mathbb{Z}}}

\newcommand{\hol}{\ensuremath{\mathcal{O}}}

\newcommand\la{\lambda}
\newcommand\Lam{\Lambda}
\newcommand\s{\sigma}
\newcommand\Si{\Sigma}

\newcommand\al{\alpha}

\newcommand\Ga{\Gamma}

\newcommand\de{\delta}

\newcommand{\ra}{\ensuremath{\rightarrow}}

\def\eea{\end{eqnarray*}}
\def\bea{\begin{eqnarray*}}

\newcommand\dual{\mathrel{\raise3pt\hbox{$\underline{\mathrm{\thinspace d
\thinspace}}$}}}
\newcommand\qe{\ifhmode\unskip\nobreak\fi\quad $\Box$}       

\def\BOX{\hfill\lower.5\baselineskip\hbox{$\Box$}}

\newtheorem{theo}{Theorem}[section]

\newtheorem{remarkk}[theo]{Remark}
\newenvironment{rem}{\begin{remarkk}\rm}{\end{remarkk}}

\newtheorem{prop}[theo] {Proposition}

\newcommand{\Proof}{{\it Proof. }}

\title [Moduli of hypersurfaces in Abelian Varieties ]{The Moduli Space of  smooth  Ample Hypersurfaces in Abelian Varieties}
\author{Fabrizio Catanese 
and Yongnam Lee
}
\address{Lehrstuhl Mathematik VIII, Mathematisches Institut der Universit\"{a}t
Bayreuth, NW II, Universit\"{a}tsstr. 30,
95447 Bayreuth, (and Korea Institute for Advanced Study, Hoegiro 87, Seoul, 
133-722, Korea)}
\email{Fabrizio.Catanese@uni-bayreuth.de}
\address {Department of Mathematical Sciences\\
KAIST\\
291, Daehak-ro, Yuseong-gu\\ 
Daejeon, 34141, Korea}
\email{ynlee@kaist.ac.kr}

\thanks{
The present work took place in the framework  of the 
 ERC Advanced grant n. 340258, `TADMICAMT'. 
 Both authors would also like to  acknowledge the support and hospitality
 of KIAS, Seoul which they visited as Research Scholar, respectively as Affiliate Professor. The second author was also supported by Samsung Science and Technology Foundation under Project Number SSTF-BA1701-04.
  A final revision of the paper was done when the first  author visited    the Mathematical Sciences Research Institute in Berkeley, during the Spring 2019 Program `Birational Geometry and Moduli', supported by the NSF  Grant No. DMS-1440140 }

\date{\today}

\begin{document}

\begin{abstract}
We  give a characterization of smooth ample Hypersurfaces  in Abelian Varieties and also  describe   the   connected component  of the moduli space containing such hypersurfaces of a given polarization type: we show that this component is irreducible and that 
it consists of these Hypersurfaces,
plus the  smooth iterated univariate coverings  of normal type (of the same polarization type).

The above manifolds yield also a connected component  of the open set of Teichm\"uller space
 consisting of K\"ahler complex structures.

\end{abstract}

\maketitle



\section {Introduction}

   Since the seminal work of Kodaira and Spencer \cite{KSII} (see also \cite{quintics})  it is clear  that the main problem in the classification theory
of complex manifolds is to describe compact complex  manifolds and the relation of deformation equivalence among them. 
For instance, any deformation of a complex torus is again a complex torus (see \cite{A-S},  \cite{deflarge}).

Via Teichm\"uller theory
\cite{superficial},
this problem  can be seen as the description of the connected components in the space of complex structures on a given oriented differentiable manifold.

In this paper we focus on the case of (smooth ample) hypersurfaces in Abelian varieties, which yield points of the moduli space
of canonically polarized manifolds introduced by Viehweg \cite{vieh}, \cite{viehweg} (or, for complex dimension 2, of the Gieseker moduli space 
\cite{gieseker} of surfaces of general type).

We know that these moduli spaces have in general a lot of distinct connected components \cite{cw}, which we cannot hope at the moment 
to fully understand; unless the manifolds in question have special topological properties, e.g.  if they have a  large fundamental group.

For instance,  in this paper we show that there is an irreducible moduli space parametrizing all the K\"ahler complex manifolds which are deformations 
of ample hypersurfaces in an Abelian variety. And we ask whether these are all the possible complex structures
on the differentiable manifolds underlying ample hypersurfaces in an Abelian variety.  

Before we turn to a more detailed description of our present results, let us recall the definition of Inoue type varieties, 
an attempt to make precise the above vague notion of `special topological properties'.

The first author and Ingrid Bauer \cite{BC-IT} defined,  quite  generally,  an Inoue-Type Variety to be the quotient $ X = X' / G$
of an ample smooth hypersurface $X'$ in a projective classifying space $Z$ by a free action of a finite group $G$.
They showed in many special cases (see  a few  instances in  \cite{BC-Bu}, \cite{BC-IT},  \cite{BCF}) that for such varieties  the topology determines the moduli space
(see also \cite{topmethods} for a general treatment of this phenomenon). But essentially, up to now, in  the theory of Inoue-type varieties
 one could describe their moduli spaces explicitly
only  in the case where 
the morphism  $\phi_0$ to the classifying variety had necessarily degree one onto its image. 

 In the case where the group $G$ is non-trivial one main difficulty is to describe the possible actions of $G$ on $Z$ or on $X'$ (as in \cite{BC-IT} or \cite{topmethods}): then one reduces   the  investigation  to the study of the $G$-action on 
the moduli space of the hypersurfaces $X'$, i.e., to the case where the group $G$ is trivial.

However, even the simplest case with $G$  trivial,
 the one where  $Z$ is an Abelian Variety (the projective classifying space of an Abelian group),
remained open.

 In this paper we  make use of a powerful result proven in  \cite{cat-lee}, characterizing the deformations of morphisms to hypersurface embeddings.
 This result is particularly suitable in order to  analyze when does the Albanese map deform  to a hypersurface embedding: because it allows us to resort to
  existing   deformation techniques  (see \cite{smalldef}, \cite{horikawaNB},  \cite{sernesi}).

Thus we can study the moduli space of compact K\"ahler manifolds diffeomorphic to ample hypersurfaces in Abelian varieties,
essentially showing that we get a connected component of the moduli space once we add to the Hypersurfaces of a given dimension
$n$, and of a given polarization type,  the iterated univariate coverings of normal type.

More precisely, we have the following main results:  theorem \ref{mr},  and theorem \ref{d=1},  characterizing smooth ample hypersurfaces in Abelian varieties.

\begin{thm}\label{mr}
Let $n \geq 2$ and $\overline{d} : = (d_1, d_2, \dots, d_{n+1})$ be a polarization type for complex Abelian Varieties.

Then,  for $ n \geq 3$, the smooth  hypersurfaces of type $\overline{d}$ in some complex Abelian variety
and the smooth Iterated Univariate Coverings of Normal Type and  of type $\overline{d}$
form an irreducible connected  component of the moduli space of canonically polarized manifolds, 
 and also an irreducible connected component of    the K\"ahler-Teichm\"uller space (of any such  smooth  hypersurface $X$).
 
 For $n=2$ we need also to include the minimal resolutions of such surfaces which have only Rational Double Points as singularities.
 
 For $d_1 =1$ there  is  only this connected component   of  the K\"ahler-Teichm\"uller space if
 \begin{itemize}
 \item
 $n=2$ or if
 \item
   we 
 restrict ourselves to compact K\"ahler manifolds $Y$ with $K_Y^n = (n+1)! d_1 d_2 \dots, d_{n+1}$, or if
 \item
  we 
 restrict ourselves to compact K\"ahler manifolds $Y$ with $p_g(Y) = d_1 d_2 \dots, d_{n+1} + n$.
 \end{itemize}
 
\end{thm} 
The K\"ahler-Teichm\"uller space is the open set of Teichm\"uller space corresponding to the K\"ahler complex structures
(see \cite{superficial} for general facts about Teichm\"uller space). In special situations (see \cite{deflarge})  it is a connected component
of Teichm\"uller space: the same should hold also in  the present case, but our arguments are not yet complete (see the proof of the main Theorem
\ref{mr}).

The results of this paper lend themselves to generalizations and   extensions    to the case of  more general Inoue type varieties;
but, for the sake of clarity, we have confined ourselves to the single crucial case of hypersurfaces in Abelian varieties.

\section{Hypersurfaces in Abelian Varieties}

Let $A$ be an Abelian variety of dimension $n+1$, and let $ X \subset A$ be a smooth and  ample divisor,
whose  Chern class is a polarization of type $(d_1, d_2, \ldots, d_n, d_{n+1})$, where $ d_i | d_{i+1}$, $\forall i=1, \dots, n$.

We assume throughout that $ n = dim (X) \geq 2$, so that Lefschetz' theorem says that
\begin{enumerate}
\item
$\pi_1(X) \cong \pi_1(A) \cong \ZZ^{2n+2} = : \Ga$;
\item
$\Lam^i (\Ga) = H^i(A, \ZZ) \ra H^i(X, \ZZ)$ is an isomorphism for $i \leq n-1$, and is injective for $i=n$;
\item
$H_i(X, \ZZ) \ra H_i(A, \ZZ)$  is an isomorphism for $i \leq n-1$, and is surjective for $i=n$.
\end{enumerate}

We consider now a projective manifold which is diffeomorphic to $X$, actually some weaker hypotheses are sufficient: 
\begin{itemize}
\item
(a) Assume that $Y$ is a complex projective manifold, or
\item
(a') Assume that $Y$ is a {\bf cKM = compact K\"ahler Manifold}, and that
\item
(b) $Y$ is homotopically equivalent to $X$, or 
\item 
(b1) there is an isomorphism $ \al_Y :  \pi_1(Y) \ra   \Ga $ and 
an isomorphism  of algebras $\psi : H^*(Y, \ZZ) \cong H^*(X, \ZZ)$ such that,  letting $\al_X : \pi_1(X)  \ra \Ga$ the analogous isomorphism,
then $ \psi \circ H^*(\al_Y) = H^*(\al_X)$; i.e., $\psi$ commutes  with the homomorphisms to $H^*(\Ga, \ZZ)$
induced by the classifying maps for $\al_Y, $ $\al_X$ respectively, or 
\item
(b2) the same occurs for homology: there are isomorphisms  $\phi_i : H_i (X, \ZZ) \ra H_i (Y, \ZZ)$ commuting with $H_i ( \al_Y), H_i ( \al_X)$, or
\item
 (b')    there is an isomorphism $ \al_Y :  \pi_1(Y) \ra   \Ga = \ZZ^{2n+2} $ such that, denoting by $a_Y$ the corresponding
classifying map, 
\item
(b'1):  $(a_Y)_* [Y]$ is dual to a polarization of type $\overline{d}$, and 
\item
(b'2):   for $ n \geq 3$, $H_2(a_Y, \ZZ)$ is an isomorphism. 
\end{itemize}

Observe that Hypothesis (a) implies (a'), Hypothesis (b) implies (b1),  (b1) implies (b2) by Poincar\'e  duality,
 and (b2) implies (b'),   (b'1) and (b'2). 

\begin{prop}\label{prel}
Assume Hypotheses (a'), (b'),  (b'1) and (b'2) above.

Then

I)  the Albanese map $ a_Y : Y \ra Alb(Y) = : A'$ has image $\Sigma$ which is an ample hypersurface, indeed
 $(a_Y)_* (Y) = deg (a_Y) \Sigma$ is the dual class of a polarization of type $(d_1, \dots, d_{n+1})$. 

II) a) holds, i.e. $Y$ is a projective manifold,

III)  if $ n \geq 3$, then $a_Y$ is a finite map.
\end{prop}

\begin{proof}
I)
 By assumption (b'1)
the class of $(a_Y)_* (Y) = deg (a_Y) \Sigma$     is the dual to a
polarization of type $(d_1, \dots, d_{n+1})$. 

II) Since $\Sigma$ is an ample hypersurface in $A'$, then $A'$ is projective: hence $\Sigma$ is projective too,
and the algebraic dimension of $Y$ equals $n$. By Moishezon's theorem \cite{moishezon}, a cKM $Y$ with algebraic dimension
equal to $dim(Y) = n$ is projective.

III)  Assume that $n \geq 3$ and that $a_Y$ is not a finite map: then there is a  curve $C$ 
 such that $a_Y (C)$ is a point.
 Then $(a_Y)_* ([C]) = 0 \in H_{2}(A', \ZZ) = H_{2}(\Ga, \ZZ)$: by Lefschetz' theorem and (b'2) follows that the 
 homology class $[C]$ is zero:
 this is impossible on a compact K\"ahler manifold.

\end{proof}

\begin{remark}
When $n = dim (Y)=2$ it can indeed happen that $a_Y$ is not finite: since we may take $\Sigma$ to be a hypersurface with Rational Double Points,
and by Brieskorn-Tyurina's theorem, the minimal resolution of singularities $Y$ is diffeomorphic to a smooth deformation $X$ of $\Sigma$.

\end{remark}

The following  characterization of smooth ample hypersurfaces in Abelian varieties is a  refinement of a theorem obtained with Ingrid Bauer \cite{BC-IT}:  in particular here the hypothesis that $K_Y$ is ample is removed:

\begin{thm}\label{d=1}
 Assume that $X$ is a smooth ample hypersurface in an Abelian variety, of dimension $n \geq 2$.

 Assume that  $Y$  is a compact K\"ahler manifold  which satisfies  the topological conditions (b'), (b'1) and (b'2)
 above.
   
Moreover, for $n \geq 3$, assume either: 

(I) $K_Y^n = K_X^n = d_1 \dots d_{n+1} (n+1)!$, or

(II) $ p_g (Y) = p_g (X) =  d_1 \dots d_{n+1} + n$. 

Whereas,  for $n=2$,  assume either the topological condition  (b1), or (I) above.

Denote the image of $a_Y$ by $W$, and 
assume either that

(i) the   class of $X$ is indivisible (i.e.,  $d_1 = 1$), or the following consequence: 

(ii) the degree of the map $a_Y :  Y \ra W$ equals $1$

Then,  for $ n \geq 3$,  
 the Albanese map $a_Y$ yields an isomorphism:
$$a_Y : Y \cong W. \ $$
Whereas, for $n=2$,  $a_Y$ is the minimal resolution of singularities of a canonical surface, i.e., a surface with 
Rational Double Points as singularities,  and with ample canonical divisor.
\end{thm}

\begin{proof}
Since the class $(a_Y)_* ([Y]) = deg (a_Y) [W]$ is indivisible, it follows  from (i)  that $a_Y : Y \ra W$ is a birational morphism,
i.e. (ii) holds.

 Moreover $Y$ is of general type with $p_g \geq n+1$, since 
$$H^0(A' , \Omega_{A'}^n) \ra H^0(Y , \Omega_Y^n) =  H^0(Y , \hol_Y (K_Y)) $$
 induces a generically finite map to $\PP^n$
(see for instance \cite{ran}).

If $n=2$, $ Y \ra W$ factors through the minimal model $Y'$ of $Y$, and indeed through the canonical model $\Xi$
of $Y$, so that we have $\pi : Y \ra \Xi, \  f : \Xi \ra W$. 

We have $$K^2_{\Xi} =  K^2_{Y'}  \geq  K^2_Y,$$
( $\Xi$ has hypersurface singularities which impose no adjoint conditions:
these are precisely the Rational Double Points, see \cite{artin})
equality holding iff $Y = Y'$, that is, iff $Y$ is the minimal resolution of singularities of  $\Xi$.

Since $K_{\Xi}$ is ample, we can argue as in \cite{BC-IT}: $K_{\Xi} = f^*(K_W) - \sA$,
where the adjoint divisor $\sA$ is effective. Since $W$ is ample, and $K_W = W |_W$,
$f^*(K_W)$ is nef and big, hence:
$$K_X^2 = K_W^2  =  f^*(K_W)^2 =  f^*(K_W) (K_{\Xi} + \sA ) \geq f^*(K_W) (K_{\Xi}) = K^2_{\Xi} + K_{\Xi} \sA \geq K^2_{\Xi},$$
equality holding if and only if $\sA = 0$.

 If (I) holds, we have 
 $$  K^2_Y = K^2_X \geq K^2_{\Xi}   \geq  K^2_Y,$$
 where equality holds if and only if $Y=Y'$ and $\sA=0$ : hence  $\Xi = W$, as claimed.
 
  If instead we use  hypothesis (b1), it implies (since we have an isomorphism of algebras) that the positivity index $b^+$ of the intersection form
 is the same for $Y$ and for $X$.
 
  Recall moreover that,  for  any algebraic surface $Y$,  $K_Y^2$ is a topological invariant,
equal to $( 6 b^+  - 4 b_1 + 4  -    b_2) $, hence  (b1)  implies that $K_Y^2 =  K_X^2 $,
i.e., (I) holds and we are done.

Before we pass to the case $n \geq 3$, recall that, on an Abelian variety $A$ of dimension equal to $g$:

\begin{enumerate}
\item
If $L$ is a line bundle,  the index theorem for complex tori states:
$$ H^i(A, L) = 0$$ for $i \notin [  \nu(L), g-p(L) ]  $, where $p(L)$ is the positivity of the Chern form of $L$,
and $\nu(L)$ is the negativity.
\item
If $L$ is effective, then there is a morphism $ f : A \ra B$, where $B$ is another Abelian variety,
and an ample line bundle $\de$ on $B$ such that $ L = f^* (\de)$.
\item
In particular, if $L$ is effective, then $L$ is nef,  $\nu(L) = 0$, and $p(L) = dim (B)$.
\item If $H$ is an ample divisor on $A$, and $ H = D+L$, where both $D,L$ are effective,
then 
$H^g \geq L^g,$
equality holding iff $D=0$.
\item 
Assume that $H= D+L$ as in (4): then 
$$H^g \geq H L^{g-1},$$
again equality holding iff $D=0$.

\end{enumerate}

{\em Proof for items 4), 5):}

Since the Chern form of $H$ is strictly positive definite, we can simultaneously diagonalize the Chern forms
for $H$ and $L$. Dividing by $H^g$, we may assume that  $H$ corresponds to the identity matrix,
while  $L$ corresponds to a   matrix  $diag (\la_1, \dots, \la_g)$.

Since $L,  D$ are  effective, $ 0 \leq  \la_i \leq 1$ $\forall i =1, \dots, g$.

Then $L^g / H^g = \la_1 \dots \la_g \leq 1$, equality iff $\la_i =1 \ \forall i =1, \dots, g$.

Similarly $ H L^{g-1}/ H^g = \frac{1}{g} \sum_i \la_1 \dots \hat{\la_i} \dots  \la_g \leq  1$, 
equality iff $\la_i =1 \ \forall i =1, \dots, g$.

 A second more general proof \footnote{ Valery Alexeev pointed out that one should rather use this more general argument}
is that if $H$ is an ample divisor on a $g$-dimensional smooth variety, and $H= D+L$, with $D,L$ nef and $D$ effective, then
$$H^g = H^{g-1} (L +D) \geq   H^{g-1} L \geq H^{g-2} L^2 \geq \dots \geq H L^{g-1} \geq L^g,$$
where the first equality holds if and only if $D=0$.

\qed

For $n \geq 3$, since $a_Y$ is finite, this map is the normalization of $W$, and it suffices to show that $W$ is normal:
then  $a_Y : Y \ra W$ is an isomorphism and $W$ is smooth.

Let us now prove the normality of $W$.

By the Lefschetz theorem and hypothesis (b'2)  (observe that since $H_1(Y, \ZZ)$ is free, $H^2(Y, \ZZ)$
has no torsion) it follows that the canonical divisor $K_Y$ is a pull-back from $ A' : = Alb(Y)$.
Hence we can write $K_Y = a_Y^* (K_W) - \sC$, where $\sC$ is the conductor divisor, and it is a pull back from 
a divisor $D$ on $A'$.

We want to show that the conductor divisor  $\sC$ is zero, whence  the normality of $W$ follows.

Step 1: the divisor $D$ is effective.

\Proof
Consider the exact sequence
$$ 0 \ra \hol_{A'} (D-W) \ra \hol_{A'} (D) \ra \hol_{W} (D)\ra 0.$$
Since the conductor is an effective divisor on $W$,  $H^0 (\hol_{W} (D)  ) \neq 0$.

Since $K_Y$ is the pull back of $W-D$, and $K_Y$ is big, hence $W-D$ has positivity at least $n$,
whence $D-W$ has negativity at least $n$ and $H^1 ( \hol_{A'} (D-W)    ) = 0$,  $H^0 ( \hol_{A'} (D-W)    ) = 0$.

Therefore  also $H^0 (\hol_{A'} (D)  ) \neq 0$.

\qed

Let us first  make assumption (II). We have, by the property of the conductor ideal, 
$$p_g(Y) = h^0 (Y, \hol_Y(K_Y)) = h^0 (Y, \hol_Y( a_Y^* (W-D)))= h^0 (W, \hol_W(  W-D )). $$

We know that $p_g(X) =  h^0 (W, \hol_W(  W))$, and we  are going to show that $p_g(X) = p_g(Y)$
if and only if $D=0$.   In fact, first of all we must have  $h^0 (W, \hol_W(  D))=1$, and 
therefore also $h^0 (\hol_{A'} (D) ) =1$.

Consider now the exact sequences 
$$0 \ra \hol_{A'} \ra  \hol_{A'} (W) \ra \hol_W(  W) \ra 0,$$
$$0 \ra \hol_{A'}(-D)  \ra  \hol_{A'} (W-D) \ra \hol_W(  W-D) \ra 0.$$

Passing to cohomology, since $W$ is ample, $H^1 (A', \hol_{A'}(  W))=0$,
and  $h^0 (W, \hol_W(  W)) = h^0 (A', \hol_{A'}(  W)) + n.$

We have  in general 
$$ h^0 (W, \hol_W  (W-D )) \leq  h^0 (A', \hol_{A'}(  W-D)) + h^1 (A', \hol_{A'}( -D)) .$$

Now, $D$ is effective, hence  by the index theorem   for complex tori, $$h^1 (A', \hol_{A'}( -D))=0$$ unless $- D$ has negativity one,
which means that $D$ pulls back from a divisor of degree $d$ on an  elliptic curve $E$.
  In the second case, since  $h^0 (\hol_{A'} (D) ) =1$, it follows that $d=1$, hence $h^1 (A', \hol_{A'}( -D))=1$. 

In both  cases, since in the second case   $h^1 (A', \hol_{A'}( -D)) = 1$, it follows that
$$ h^0 (W, \hol_W  (W-D)) \leq  h^0 (A', \hol_{A'}(  W-D)) + 1   \leq  h^0 (A', \hol_{A'}(  W)) +1$$
and we get  a  contradiction to $p_g(X) = p_g(Y)$.

\medskip

Let us make now assumption (I), namely, that $K_Y^n = K_X^n ( = K_W^n)$.

Step 2: the divisor $L : = W-D$ is effective or $D^2=0$.

\Proof

Consider the exact sequence

$$ 0 \ra \hol_{A'}(  - D ) \ra \hol_{A'}(  W - D ) \ra \hol_{W}(  W - D) \ra 0,$$
where   $ h^0 ( \hol_{W}(  W - D) ) \geq n+1$
since, as we have already observed at the beginning of the proof,    $ h^0 ( \hol_{W}(  W - D) ) = p_g(Y)$,
and, by assumption (II), $p_g(Y) = p_g(X) \geq n+1$.

It follows that $W-D$ is effective on $A'$ provided  $h^1 (A', \hol_{A'}( -D)) = 0$.

In the contrary case, $D$ is again a pull-back from an elliptic curve  hence $D^2 = 0$.

\qed

 If $W-D$ is effective we apply then item (5):
$$ K_X^n = K_W^n = W^{n+1} \geq W (W-D)^n = K_Y^n.$$
Since equality holds by assumption I, it follows then that $D=0$, i.e., $W$ is normal and isomorphic to $Y$.

 Otherwise, $D^2=0$ and $W (W-D)^n = W^{n+1} - n W^n D \leq W^{n+1}$, with equality holding if and only if $D=0$.

\end{proof}

\begin{remark}
If the dimension $n \geq 3$, then the geometric genus of a compact complex manifold is not a differentiable
invariant, as shows the case of Jacobian Blanchard-Calabi manifolds, cf. remark 7.3 of \cite{deflarge}.

 Similarly, Le Brun \cite{lebrun} gave examples of complex threefolds for which $K_X^3$ is not a differentiable invariant.

 Moreover, Kotschick has proven in  \cite{kotschick}  that, for smooth projective manifolds $X$
of dimension $n \geq 3$,  $K_X^n$ is not a differentiable invariant.

\end{remark}

\begin{remark}\label{ss}

 An essential step in the above proof is to show  the normality of $W$: for  a hypersurface in a smooth manifold, normality is equivalent to the condition
that the singular locus of $W$ has codimension at least 2. 

In turn, this is equivalent to asking that a general  curve $C$ obtained as the intersection of $W$ with $n-1$ 
very ample hypersurfaces $H_1, \dots, H_{n-1}$ is smooth. One could 
for instance  take these Hypersurfaces to lie in the class of $3 W$.

Since the system $| a_Y^* (3W)|$  is base point free on $Y$, by Bertini's theorem the inverse image of $C$ in $Y$ is a smooth curve $D$. Since $C$ is reduced, and $ f :  D \ra C$ is birational, the cokernel of the pull back map $  \hol_C \ra f_* \hol_D$
is  a skyscraper sheaf $\sF$: if $D$ and $C$ were to be shown to have the same arithmetic genus $p(C) = p(D)$, then $\chi(\sF) = 0$, hence
$\sF = 0$ and we could conclude that $ D \cong C$ , hence $C$ is smooth. 

$p(C) = p(D)$ follows if one knows that the class of the canonical divisor of $Y$ corresponds to the class of the
canonical divisor of $X$: since then, by adjunction and  by hypothesis (b''), the genus of $D$
 equals the one of the analogous intersection curve $C' \subset X$,
 which in turn equals the arithmetic genus of $C$. 
 
A similar  argument was given  in lemma 1.5 of  \cite{o-s}: but, as we have shown,  this argument in our case
requires a stronger condition than $K_Y^n = K_X^n$, namely  that $\psi ([K_Y]) = [K_X]$  (see hypothesis
(b1), where the isomorphism $\psi$ was introduced).
\end{remark}

\section{Iterated univariate coverings of normal type = ITUNCONT}

For the reader's benefit, we  recall  the following definitions, introduced in \cite{cat-lee}.

\hfill\break
i) Given a complex space (or a scheme) $X$, a {\bf univariate covering} of $X$ is a hypersurface $Y$, contained in a line bundle
over $X$, and defined there as the zero set of a monic polynomial. 

In more abstract wording, $ Y = \underline{Spec} (\sR)$, where $\sR$ is the quotient algebra of the symmetric algebra 
over an invertible sheaf $\sL$, $$ Sym (\sL) = \oplus_{i \geq 0} \sL^{\otimes i}  ,$$ by the principal ideal generated 
by a monic (univariate) polynomial $P$.

 Indeed, $Sym (\sL)$ is a sheaf of univariate (one variable) polynomials in $w$, where $w$ is the tautological  linear form
on   $Spec (Sym (\sL))$, a line bundle over $X$. In more concrete terms,
$$ \sR : = Sym (\sL) / (P) , P = w^m+ a_1(x) w^{m-1} +a_2(x) w^{m-2}+\cdots +a_m(x).$$
Here $a_j \in H^0 (X, \sL^{\otimes j})$.

Without loss of generality, over $\CC$ one may assume the covering to be in Tschirnhausen form, that is, with $a_1 \equiv 0$.

 The univariate covering is said to be {\bf smooth} if both $X$ and $Y$ are smooth.

ii) An {\bf   iterated univariate covering} $ W \ra X$ is a composition of  univariate coverings
$$f_{k+1} : X_{k+1}: = W \ra X_{k}, f_k : X_k \ra X_{k-1}, \dots , f_1 : X_1 \ra X =: X_0,$$ whose associated line bundles are
denoted $\sL_k,  \sL_{k-1},  \dots, \sL_1,   \sL_0 $.

It is said to be smooth if all $f_j$ are smooth.

iii) In the case where $ X \subset Z$ is a (smooth) hypersurface, we say that the iterated univariate covering
is of {\bf normal type} if 

\begin{itemize}
\item
all the line bundles $\sL_j$ are pull back  from $X$ of a line bundle
of the form $\hol_X ( m_j X)$  and moreover
\item
$ 1 = m_0 | m_1 | m_2  \dots | m_k | m_{k+1} := m $,
 and the degree of $f_j$ equals  $\frac{m_j}{m_{j-1}}$.
\item
we say that the iterated covering is {\bf normally induced} if moreover 
  all the coefficients $a_I(x)$ of the polynomials
  
$$ Q_{j}(w_0, \ldots, w_{j-1},x)= \sum_I  a_I(x) w^I $$

describing the intermediate extensions are sections of a line bundle $\hol_X(r(I)X)$ coming from
$H^0 (Z, \hol_Z(r(I)X))$.
\end{itemize}

\begin{remark}\label{ituncont}
(i) The property that the iterated univariate covering $ W \ra X$ is normally induced clearly means that
it is the restriction to $X$ of an iterated univariate covering of $Z$.

(ii) For simplicity of notation we shall use the same notation  $\sL_j$,  $j=0, \dots, k+1$, to denote
the pull back  from $X$ of the line bundle
 $\hol_X ( m_j X)$  via any
 iterated intermediate covering.
 
  (iii) As already done in the heading of this section, we shall use the acronym ITUNCONT to 
 refer to the iterated univariate coverings of normal type.

\end{remark}

This said, we proceed to construct iterated univariate coverings of hypersurfaces in an Abelian variety.

\begin{defin}
Given a polarization type $\overline{d} = (d_1, d_2, \ldots, d_n, d_{n+1})$ (here, as before, $ d_i | d_{i+1}$, $\forall i=1, \dots, n$)
and a sequence of positive  integers 
$ \overline{m} = (m_0 = 1, \ m_1 | m_2 | \dots | m_{k+1}),$ such that $m_{k+1}$ divides $d_1$, we define the family
 $\sF_{\overline{d} ,\overline{m}}$   of  iterated univariate coverings of normal type,
and of numerical-bitype $\overline{d} ,\overline{m} $, of hypersurfaces in an Abelian variety,  as the  family 
 of the following complete intersections $W$:

\begin{itemize}
\item
there is an Abelian variety $A$ with a polarization $\sL_0,$ of type $\frac{1}{m_{k+1}} \overline{d}$,
and a smooth divisor  $X $ in $|\sL_0|$
\item
 there are   line bundles $\sL_0, \ldots, \sL_k$ defined as  $\sL_j :  = \hol_{A} ( m_j X)$ for $j=0, \ldots, k$,
\item
$W$ is  a complete intersection in  the vector bundle associated to $\sL_0 \oplus\cdots\oplus\sL_{k} $,
defined by equations of the  following standard form:
\begin{equation}\label{steq2}
\begin{cases}
\s(z)=w_0t_0\\
Q_1(w_0, z)=w_1t_1 \\
\cdots\, \,\,\,\, \cdots \\
Q_{k}(w_0, \ldots, w_{k-1},z)=w_{k}t_k\\
Q_{k+1}(w_0, \ldots, w_k, z )=0.
\end{cases}
\end{equation} 
\item 
where $t_0, t_1 , \dots t_k \in \CC$, $\s \in H^0(A, \sL_0)$, $div (\s) = X$, and
$$ Q_{j}(w_0, \ldots, w_{j-1},z) = $$
$$ w_{j-1}^{\frac{m_j}{m_{j-1}}} + a_2(w_0, \ldots, w_{j-2},z) w_{j-1}^{\frac{m_j}{m_{j-1}} -2 } +\cdots +a_{ m_j/m_{j-1}}(w_0, \ldots, w_{j-2},z)$$
\item
 so that  for $t_0= t_1 = \dots =  t_k = 0$ the  projection of $W_0$ onto $X$ is 
  a normally induced   iterated  smooth univariate covering,
\item
 whereas for $t_0 \neq 0,  t_1 \neq 0,  \dots ,   t_k \neq  0$, we get a hypersurface.

\end{itemize}

\begin{thm}\label{open}
1) The family $\sF_{\overline{d} ,\overline{m}}$  of  iterated univariate coverings of normal type,
and of numerical-bitype $\overline{d} ,\overline{m} $,
of hypersurfaces in an Abelian variety  is a  family with smooth base which is locally complete,
i.e., such that at each point $W$ it maps to an open set of the base $Def(W)$ of the
Kuranishi family of $W$.  

2) The family includes all smooth hypersurfaces of polarization type  $\overline{d} $.
\end{thm}

\begin{proof}
Abelian varieties with a given polarization type are parametrized by the Siegel upper half space $\sH_{n+1}$,
and form a family $\sA_{\overline{d}, \overline{m}}$ over $\sH_{n+1}$,  endowed with a universal bundle $\mathfrak L_0$
yielding on each fibre a polarization of type $\frac{1}{m_{k+1}} \overline{d}$.
Moreover, for an ample line bundle $\sL$ on an Abelian variety $A$, the dimension $h^0(\sL)$ only depends on its polarization type:
therefore  we take as base space for $\sF_{\overline{d} ,\overline{m}}$  a smooth complex manifold $\sB_{\overline{d} ,\overline{m}}$,  a dense open set in
a vector bundle  over the family $\sA_{\overline{d}, \overline{m}}$ (the latter parametrizes the line bundles of a given polarization type
 as points in the fibres of $\sA_{\overline{d}, \overline{m}}$ transform  $\mathfrak L_0$ via translation) .

If $W$ is a variety in the family $\sF_{\overline{d} ,\overline{m}}$, then  by definition of the family $W$ is smooth and we consider the deformation of the map
$ f : W \ra A$. The map $f$ is the Albanese map of $W$:  for hypersurfaces $W \subset A$ this follows  from the theorem of Lefschetz.

 For the other ITUNCONT 's  (see remark \ref{ituncont}) we use that the Hodge number $h^{1,0}(W)$  and the first homology group are
 deformation invariants.

A deformation  of the complex manifold $W_s, s \in \sS$, of $W$ yields a morphism 
$$ \Phi : \sW \ra \sA, \ A_s = Alb (W_s)$$
 that gives the family of Albanese maps: 
since the image of $W_s$ has dimension $n$,  automatically we get  a deformation   $A_s$ of the  Albanese variety $A$ of $W$,
together with  an ample divisor  in $A_s$. Hence  we see that we get necessarily a deformation of $A$  as a polarized Abelian variety.

Since $\sB_{\overline{d} ,\overline{m}}$ maps submersively onto Siegel space, it suffices to show the completeness of the deformations
with fixed target $A$.

In this case the tangent space to the deformations (with fixed target) is given by $H^0(W, N_f)$, where $N_f$
is the normal sheaf to the map $f$, defined by Horikawa \cite{horikawaNB} as the cokernel of
$$ 0 \ra T_W \ra f^* T_A = \hol_W^{n+1} \ra N_f \ra 0. $$ 

 The exact cohomology sequence yields ($0 = H^0( T_W)$ since $W$ is of general type):

$$ 0  \ra \CC^{n+1}  \ra H^0 (N_f ) \ra H^1(T_W)  \ra   H^1(\hol_W)^{n+1} =  H^1( T_A) ,$$ 
where the first map corresponds to the action of $A$  via translations, and the last
map yields the deformation of the Abelian variety $A$.

We want to show that it suffices to show 1) for the case where $W$ corresponds to the choice $t_0= t_1 = \dots =  t_k = 0$
of the parameters $t_j$. 

 In passing, we shall prove 2).

{\bf 2) holds:} this is elementary, it suffices to fix a hypersurface of equation $\Phi(z) = 0$, and in the above equations
to take all $t_j \neq 0$, and finally fixing the constant term in the polynomial
$a_{ m_{k+1}/m_{k}}(w_0, \ldots, w_{k-1},z)$ so that, replacing inductively 
$w_j$ by $ \frac{1}{t_j} Q_{j}(w_0, \ldots, w_{k-1},z)$ we get from $Q_{k+1}$ the desired equation $\Phi(z)$.

Similarly, for each $t_j \neq 0$, we can eliminate $w_j$ and reduce to an iterated cover with a smaller number $k$
of iterated coverings.
Hence, it suffices to prove 1) in the case where all $t_j=0$.

As already observed, since $W$ is of general type, $H^0(T_W) = 0$; moreover  the image of  $H^0(f^* T_A ) = \CC^{n+1}$ accounts for deforming the class
of the line bundle $\sL_0$ in $Pic (A)$ (via translations on $A$).

Therefore it suffices to look at  the deformation of the map $ f $ into the vector bundle 
associated to $\sL_0 \oplus\cdots\oplus\sL_{k} $.

Here, we shall just show that the subfamily where we vary only the equations $Q_j$ surjects onto the
normal bundle of $W$, which is indeed the restriction of 
 $\sL_0 \oplus\cdots\oplus\sL_{k} \oplus\sL_{k+1}$ to $W$  since $W$ is a complete intersection.

For this, consider the chain of maps 
$$f_{k+1} : X_{k+1} : = W \ra X_{k}, f_k : X_k \ra X_{k-1}, \dots , f_1 : X_1 \ra X ,$$
and set $\frac{m_j }{ m_{j-1}} = : h_j$.

We have $(f_{k+1})_* (\hol_W) = \oplus_{i=0}^{h_{k+1} -1} \hol_{X_{k}} ( - i \sL_{k} )$,
hence 
 $$H^0( W, \sL_j) = H^0( X_k, \sL_j)$$  for $ j  <  k $ and 
$$H^0( W, \sL_k) = H^0( X_k, \sL_k) \oplus \CC.$$

Proceeding inductively, we see in a similar way  that 
$$H^0( W, \sL_j) = H^0( X_{j}, \sL_j) \oplus \CC.$$

And by the same argument 

$$ H^0( X_{j}, \sL_j)  = \{  Q_{j}(w_0, \ldots, w_{j-1},x) =   \sum_I  a_I(x) w^I \}, $$
where $ a_I(x) \in H^0 ( \hol_X(r(I)X))$, whereas $t_j w_j$ accounts for the summand $\CC$.

 It suffices finally to show that these sections come   from
$H^0 (Z, \hol_Z(r(I)X))$.  The surjectivity  of  
$H^0 (A,\hol_{A }(iX)) \ra H^0 (X, \hol_{X} (iX))$ for $ i\geq 2$ is  
 implied by  $H^1 (A,\hol_{A }(iX)) = 0, \forall i \geq 1,$
 as follows from the exact sequence
 $$ 0 \ra  \hol_{A} (i X) \ra \hol_{A }((i+1)X) \ra \hol_{X }((i+1)X) \ra 0  .$$

 We do not need the surjectivity for $i =  1$, as  we simply recall that all the sections of the normal bundle give a deformation which,
after an automorphism replacing $w_j$ with $ w_j - \frac{1}{h_j}  a_{1,j} (x)$,
can be put  in Tschirnhausen form, i.e., with $a_{1,j} (x) \equiv 0$.

\end{proof}

\end{defin}

\section{Deformations to hypersurface embeddings}

\begin{defin}
{\bf A 1-parameter deformation to hypersurface embedding} consists of the following data:

\begin{enumerate}
\item
 a one dimensional family of smooth projective varieties of dimension $n$ (i.e., a smooth
 projective holomorphic map $p :  \sW \ra T$ where $T$ is   a germ  of a smooth
holomorphic  curve at a point $0 \in T$)
mapping to another  family $\pi: \sZ  \ra T$ of smooth projective varieties of dimension $n+1$ via  a relative map $\Phi : \sW \ra \sZ$  such that $\pi \circ \Phi = p$
and
such that moreover 
\item

for $t\ne 0$ in $T$, $\Phi_t$ is an embedding of $W_t : = p^{-1}(t)$ in $Z_t$,
\item the restriction of the map $\Phi$ on $W_{0}$ is a generically finite morphism of degree $m$,   so that the image of $\Phi|_{W_{0}}$ is the cycle $\Si_{0}:=mX$ where  $X$ is a reduced  hypersurface in $Z_{0}$, defined by an equation
$X =\{\s=0\}$.
\end{enumerate}

\end{defin}

The following is the first part of the main theorem of \cite{cat-lee} (without the converse part):

\begin{theorem} \label{MT} (A) Suppose we have a 1-parameter deformation to hypersurface embedding  
and assume  that $K_{W_{0}}$ is ample.

Then we have:

(A1) $X$ is smooth,

(A2)  There are line bundles $\sL_0, \ldots, \sL_k$ on $\sZ$, such that  $\sL_j | _{Z_{0}} = \hol_{Z_{0}} ( m_j X)$ for $j=0, \ldots, k$,  with
  $ 1 = m_0 | m_1 | m_2  \dots | m_k | m_{k+1} := m $  (here $m$ is the degree of the morphism $\Phi_0 : W_0 \ra X$),    and such that 
  $W_{0}$ is  a complete intersection in  the vector bundle associated to $\sL_0\oplus\cdots\oplus\sL_{k}| _{ Z_{0} }$, with $\Phi_0$
  a normally induced   iterated  smooth univariate covering .

(A3)  $\sW$ is obtained  from  $\Sigma : = \Phi (\sW)$  by a finite sequence of blow-ups.

 Moreover  the local equations of $\sW$ are  of the  following standard form 
\begin{equation}\label{steq2}
\begin{cases}
\s(z)=w_0t\\
Q_1(w_0, z)=w_1t \\
\cdots\, \,\,\,\, \cdots \\
Q_{k}(w_0, \ldots, w_{k-1},z)=w_{k}t\\
Q_{k+1}(w_0, \ldots, w_k, z, t)=0.
\end{cases}
\end{equation}

\end{theorem}

The following   intermediate result played an important role in the proof.

 \begin{lemma} \label{smooth}
Suppose we have a one dimensional smooth family $p :  \sW \ra T$ of smooth projective varieties of dimension $n$ mapping to another  flat family $q: \sY  \ra T$ of  projective varieties of the same dimension via  a relative map $\Psi: \sW \ra \sY$ over a smooth
holomorphic  curve $T$ such that $q \circ \Psi = p$.
  
Assume that
\begin{enumerate}
\item $\sY$ is normal and Gorenstein,
\item $\Psi$ is birational,
\item for $t\ne 0$ in $T$, $\Psi$ induces an isomorphism,
\item $K_{W_{0}}$ is ample.
\end{enumerate}
Then we have  that $\Psi$ is an isomorphism, in particular $W_{0} \cong Y_{0}$.

\end{lemma}

\begin{proof}

We have $K_{\sW}=\Psi^*(K_{\sY}) + \sB$. Since we assume that $\Psi$ induces an isomorphism for $t \ne 0$ in $T$, the support of 
the Cartier divisor $\sB$ is contained in $W_{0}$, which is  irreducible.

Now  $Y_{0}$ has dimension $n$ and the morphism  $\Psi_0 : W_{0} \ra Y_{0}$ is generically finite, hence we conclude that $\sB=0$.
In particular, $K_{\sW}=\Psi^*(K_{\sY})$; restricting to the special fibre, we obtain  $K_{W_{0}} =\Psi_0 ^*(K_{Y_{0}})$. Since 
by assumption $K_{W_{0}}$ is ample, we obtain that $\Psi_0$ is finite, hence also $\Psi$ is finite, whence an isomorphism in view of the normality of $\sY$.

\end{proof}

\begin{rem}\label{oss} Without the  assumption that   $K_{W_{0}}$ is  ample one can only assert that $Y_{0}$ is normal with at most canonical singularities.

\end{rem}

\begin{prop}\label{ample}
Suppose we have a 1-parameter deformation to hypersurface embedding, where $W_t$ is a hypersurface in an Abelian variety, 
and $W_0$ is K\"ahler.

 If $ n \geq 3$, then necessarily $K_{W_{0}}$ is ample.
\end{prop}

\begin{proof}
By proposition \ref{prel} it follows first of all  that $W_0$ is projective.

Rerunning the proof of theorem \ref{MT} we construct, 
 by a finite sequence of blow-ups 
starting from  $\Sigma : = \Phi (\sW)$, 
a family $q : \sY \ra T$ as in  lemma \ref{smooth},
and with relatively ample canonical divisor,
such that   the local equations of $\sY$ are  of the  following standard form 
\begin{equation}\label{steq2}
\begin{cases}
\s(z)=w_0t\\
Q_1(w_0, z)=w_1t \\
\cdots\, \,\,\,\, \cdots \\
Q_{k}(w_0, \ldots, w_{k-1},z)=w_{k}t\\
Q_{k+1}(w_0, \ldots, w_k, z, t)=0.
\end{cases}
\end{equation}

 Since  $ n \geq 3$, again proposition  \ref{prel} guarantees  
that $\Psi_0$ is finite. Hence, as in  the proof of lemma \ref{smooth},   $K_{W_{0}} = \Psi_0^* (K_{Y_0})$ is ample.

\end{proof}

\section{Proof of the main theorem \ref{mr} }

For the first assertion, theorem \ref{open} ensures that ITUNCONT 's yield an open set in Teichm\"uller space; this open set is irreducible because the dense open set corresponding to smooth hypersurfaces in Abelian varieties is irreducible.

We would like  now to show that  this set is closed.

We observe that every variety $X_0$ in the closure will be the limit of a 1-parameter deformation. 

First of all, we can use  some results of \cite{deflarge}, namely lemma 2.4 and corollary 2.5, asserting that $X_0$ will have a very good Albanese variety,
and that the image of the Albanese map will be a hypersurface. Hence the algebraic dimension of $X_0$ is equal to its dimension.

 At this point we use the  assumption that $X_0$ is K\"ahler for $ n \geq 3$: for $n=2$ it follows automatically since $X_0$ is then of general type.

In dimension $ n \geq 3$, as observed in proposition \ref{ample},  the canonical bundle will be ample since the Albanese map is finite.

We can then apply Theorem \ref{MT} to conclude.

For $n=2$ we work directly with the canonical models of our surfaces, and everything works verbatim,  by remark \ref{oss}.

The case $d_1 =1$ is just a restatement of theorem \ref{d=1}: 
 observe just  that homeomorphism implies that the Betti numbers and $b^+$ are the same,
hence also $p_g, K^2$.

\section{Final question}

It is natural to ask whether our irreducible connected component  is the unique one.
To prove this, it would of course be sufficient to show that the Albanese map can be deformed to an embedding.

A second question is whether, for $n \geq 3$, we really get a connected component of Teichm\"uller space: for this it would be
sufficient to show that the Albanese map for  a limit variety $X_0$  is finite (equivalently, that $X_0$ is projective).  

\medskip

 {\bf Acknowledgements:} the first author would like to  thank Stefan Schreieder for an interesting conversation
  and especially Ingrid Bauer for spotting a mistake in the first version of theorem \ref{d=1}, and for 
 several very useful discussions. Thanks to the referee for careful reading of the manuscript, which helped remove 
 a few inaccuracies.


\end{document}